\documentclass[12pt,letterpaper,titlepage]{amsart}
\usepackage{amsmath, amssymb, amsthm, amsfonts,amscd,amsaddr}

\theoremstyle{plain}
\newtheorem{theorem}{Theorem}[section]
\newtheorem{proposition}[theorem]{Proposition}
\newtheorem{cor}[theorem]{Corollary}

\newtheorem{prop}[theorem]{Proposition}
\newtheorem{lemma}[theorem]{Lemma}

\theoremstyle{definition}

\newtheorem{rmk}[theorem]{Remark}
\numberwithin{equation}{section}
\newtheorem*{theoremA*}{Theorem A}
\newtheorem*{theoremB*}{Theorem B}
\newtheorem*{theoremm1*}{Theorem A'}
\newtheorem*{theoremC*}{Theorem C}
\newtheorem*{theoremD*}{Theorem D}
\newtheorem*{theoremE*}{Theorem E}
\newtheorem*{theoremF*}{Theorem F}
\newtheorem*{theoremE2*}{Theorem E2}
\newtheorem*{theoremE3*}{Theorem E3}

\newcommand{\bs}{\backslash}

\newcommand{\C}{\mathbb{C}}

\newcommand{\Q}{\mathbb{Q}}
\newcommand{\Z}{\mathbb{Z}}
\newcommand{\Yc}{\mathcal{Y}}
\newcommand{\Zc}{\mathcal{Z}}

\newcommand{\R}{\mathbb{R}}
\newcommand{\N}{\mathbb{N}}

\newcommand{\End}{\operatorname{End}}
\newcommand{\Hom}{\operatorname{Hom}}

\newcommand{\Ad}{\operatorname{Ad}}

\newcommand{\re}{\operatorname{Re}}

\newcommand{\Tr}{\operatorname{Tr}}

\usepackage[usenames]{color}

\def\af{\mathfrak{a}}
\def\e{\epsilon}
\def\gf{\mathfrak{g}}

\def\hf{\mathfrak{h}}

\def\kf{\mathfrak{k}}
\def\lf{\mathfrak{l}}
\def\mf{\mathfrak{m}}
\def\nf{\mathfrak{n}}

\def\pf{\mathfrak{p}}
\def\qf{\mathfrak{q}}

\def\sf{\mathfrak{s}}

\def\uf{\mathfrak{u}}

\def\zf{\mathfrak{z}}
\def\la{\langle}
\def\ra{\rangle}
\def\1{{\bf1}}

\def\U{\mathcal{U}}

\usepackage[usenames]{color}

\title[Multiplicity bounds]
{Multiplicity bounds and the subrepresentation theorem for real spherical spaces}

\subjclass[2000]{22E45, 43A85}
\keywords{Homogeneous space, Harish-Chandra module, subrepresentation theorem}

\begin{document}
\date{March 15, 2014}

\begin{abstract}
Let $G$ be a real semi-simple Lie group
and $H$ a closed subgroup which 
admits an open orbit on the flag manifold
of a minimal parabolic subgroup. Let $V$ be a Harish-Chandra
module. A uniform finite bound is given for the 
dimension of the space of
$H$-fixed distribution vectors for $V$ and
a related subrepresentation theorem is derived.
\end{abstract}
\author[Kr\"otz]{Bernhard Kr\"otz}
\email{bkroetz@math.uni-paderborn.de}
\address{Universit\"at Paderborn\\Institut f\"ur Mathematik\\Warburger Stra\ss e 100\\
D-33098 Paderborn\\Germany}
\author[Schlichtkrull]{Henrik Schlichtkrull}
\email{schlicht@math.ku.dk}
\address{University of Copenhagen\\Department of Mathematics\\Universitetsparken 5\\
DK-2100 Copenhagen \O\\Denmark}

\maketitle

\section{Introduction}

Let $G$ be a connected real semi-simple Lie group and $P = MAN$
a minimal parabolic subgroup. Let $H <G$ be a closed 
and connected subgroup.  We call $Z=G/H$ {\it real spherical} provided 
that there is an open $P$-orbit on $Z$. 
In \cite{KS} we have shown that 
this condition implies that there are only finitely many $P$-orbits 
on $Z$. The purpose of this paper is to explore the representation theoretic 
significance of real sphericity. 

\par This paper relies in part on the forthcoming article \cite{KKS} on the local structure 
of real spherical spaces, which forces us to assume that the Lie algebra
of $H$ is an algebraic subalgebra of the Lie algebra of $G$. 

\par For a Harish-Chandra
module $V$ with smooth completion $V^\infty$, we show that 
if $G/H$ is real spherical with $PH$ open then
\begin{equation}\label{u-bound} \dim \mathrm{Hom}_H
(V^\infty, \C) \leq \dim (V/\bar\nf V)^{M\cap H}\, ,\end{equation}
with  
$\bar\nf ={\rm Lie}(\bar N)$ corresponding to an
opposite parabolic subgroup.

In this context we recall that $V/\bar\nf V$ is  finite dimensional,  a
consequence of the Casselman-Osborne lemma (see \cite{W}, Sect.~3.7).
For symmetric spaces  (which  
are real spherical) finite dimensionality of 
$\mathrm{Hom}_H (V^\infty, \C)$ was originally established by van den Ban 
in \cite{vdB2}, Cor.~2.2. 
Finally we remark that certain bounds on $\dim \mathrm{Hom}_H
(V^\infty, \C)$ were obtained with a different technique by Kobayashi and Oshima in \cite{KO}, Thm.~2.4. 

The bound in (\ref{u-bound}) is essentially sharp as equality 
is obtained for $H=\bar N$ and generic irreducible representations $V$.
However, a statement is presented in Thm.~\ref{thm1} which
in general can be stronger than (\ref{u-bound}).

The main part of the proof of (\ref{u-bound}) is elementary in the sense that it 
only invokes simple methods of ordinary differential equations, 
applied to generalized matrix coefficients 
on $Z$ (cf.~the proof of the subrepresentation theorem in \cite{W}, 
Sect.~3.8).  However these methods 
typically result in asymptotic expansions only. To prove that a matrix coefficient which 
is asymptotically zero (i.e.~of super-exponential decay), is in fact vanishing, we need more
elaborate analytic methods. 
This is done in the end of the paper where we adapt some results from
\cite{CM} and \cite{vdB} to show that the relevant system of differential equations
has a regular singularity at infinity.
 
\par Suppose that $PH$ is open. According to \cite{KKS} there 
exists a parabolic subgroup  $Q\supset P$ with a Levi decomposition $Q=LU$ 
such that $Q\cap H\subset L$ and $L/ (Q\cap H) Z(L)$ is compact.

In case that $V$ is irreducible and $H$-spherical, i.e.
$\mathrm{Hom}_H(V^\infty, \C)\neq \{0\}$, we prove 
the spherical subrepresentation theorem which asserts that 
$V$ is a submodule of an induced module $\mathrm{Ind}_{\bar Q}^G \tau$, 
where $\tau$ is an irreducible finite dimensional
representation of $L$ which is 
$L\cap H$-spherical. This was established for symmetric spaces by Delorme
in \cite{Patrick}. For the group, it is the subrepresentation theorem of
Casselman~\cite{cass}.

\section{Some structure theory for real spherical spaces}\label{str}

Let $G$ be a real reductive group and $H<G$ a closed subgroup 
with finitely many connected components. In what follows Lie algebras are 
always denoted by corresponding lower case German letters, i.e. $\gf$ is 
the Lie algebra of $G$, $\hf$ the Lie algebra of $H$ etc.

\par Let $P<G$ be a minimal parabolic subgroup of $G$. The homogeneous space 
$Z=G/H$ is called {\it real spherical}, provided that $P$ admits open orbits 
on $Z$. This means that by replacing $P$ with a conjugate we can obtain 
that $PH$ is open, that is 
$$\gf = \pf +\hf\, .$$

If $\lf$ is a reductive 
Lie algebra, then we denote by $\lf_n$ the sum of all simple non-compact ideals 
of $\lf$. We recall the following result from \cite{KKS}.

\begin{prop}\label{lst}
Let $H\subset G$ be algebraic groups over $\R$, and assume
$Z=G/H$ is real spherical. Let $P$ be a minimal parabolic subgroup such that
$PH$ is open.
Then there exists a parabolic subgroup $Q$ such that 
\begin{enumerate}
\item\label{eins} $Q\supset P$.
\item\label{zwei} There is a Levi-decomposition $Q=LU$ such that 
$$\lf_n\subset  \qf \cap \hf \subset \lf\, ,$$
and in particular, $L/(L\cap H)Z(L)$ is compact 
(here $Z(L)$ denotes the center of $L$).
\item\label{vier} $L\cap H$ has finitely many components.
\item\label{drei} $QH=PH$
\end{enumerate}
\end{prop}

Through most of the paper we only need the properties (\ref{eins})-(\ref{zwei}), and for that 
it clearly suffices to assume that $Z$ is real spherical and 
{\it locally} algebraic, 
that is, there are real algebraic groups $H_1\subset G_1$ with 
Lie algebras $\hf$ and $\gf$. 
We assume throughout that $Z$ is locally algebraic, but emphasize that
this assumption is only used to obtain (\ref{eins})-(\ref{zwei}) above. 
 In Remarks \ref{Lcap H-rep} and \ref{second remark}  
we need also (\ref{vier}), which is a consequence if
$G$ and $H$ are algebraic since then $L$ is algebraic.
The last property (\ref{drei})
is important for \cite{KKS} but will not be
needed here.

Note that in case $H$ is symmetric, then the minimal $\sigma\theta$-stable 
parabolic subgroups (\cite{vdB3} Sect.~2) 
satisfy   (\ref{eins})-(\ref{drei}).
Here $\theta$ is a Cartan involution which commutes 
with the involution $\sigma$ which defines $\hf$. 

\par If $G/H$ is real spherical and $P\subset Q=LU$
is as above, we let $\theta$ be a Cartan involution of $G$
which leaves $L$ stable. The existence of $\theta$ follows 
since $L$ is a reductive subgroup of $G$.
Let $\gf=\kf+\sf$ be the Cartan decomposition and
$K\subset G$ the corresponding maximal compact subgroup,
then $K_L=L\cap K$ is maximal compact in $L$.
Let $\af$ be a maximal abelian subspace in $\lf\cap\sf$.
We may assume $A$ is contained in $L\cap P$,
since this intersection is a minimal parabolic subgroup in $L$.
Let $L=K_L A N_L$  be an Iwasawa decomposition of $L$ and 
put $N=N_LU$. Note that $\af$ is maximal abelian in $\sf$ as well. Let $M$ be the centralizer 
of $A$ in $K$. Then $P=MAN$, and it follows from (\ref{zwei}) above that 
$\af=\af\cap\zf(\lf)+\af\cap\hf.$
Let $\af_Z \subset \af\cap \zf(\lf)$ be a vector space complement to
$\af\cap\hf$,
$$\af=\af_Z\oplus(\af\cap\hf),$$
and $\mf_Z\subset \mf$ a subspace such that
$\af_Z+\mf_Z$ complements $(\af+\mf)\cap\hf$ in $\af+\mf$, 
then we arrive at the direct sum 
decomposition 
\begin{equation}\label{ldeco} \gf = \hf \oplus \af_Z \oplus \mf_Z \oplus\uf\, .\end{equation}

Let $L_n$ be the analytic subgroup of $L$ with Lie algebra $\lf_n$.
Since $\lf=\mf+\af+\lf_n$ and $\lf_n\subset \hf$, and since $M$ meets every component of $L$ 
(see \cite{KnappBeyond}, Prop. 7.33),
we conclude
\begin{equation*}%\label{L dec}
L=MAL_n
\end{equation*}
and $L_n\subset H$.
For any Lie group $J$ we denote by $J_0$ 
its identity component. 

\begin{lemma}\label{LcapH} There exists
a vector subgroup $A_h\subset MA$ such that 
\begin{equation}\label{identity component LcapH}
(L\cap H)_0=(M\cap H)_0A_hL_n.
\end{equation}
Moreover,
if $L\cap H$ has finitely many connected components, then 
\begin{equation}\label{all of LcapH}
L\cap H=(M\cap H)A_hL_n,
\end{equation}
In this case $L\cap H$ is a real reductive group and 
$L\cap H\cap K$ is a maximal compact subgroup.
\end{lemma}

\begin{proof} We have $L\cap H=(MA\cap H)L_n$. 
Since the Lie algebra $\mf+\af$ is compact, then so is its intersection with $\hf$.
It follows that $(MA\cap H)_0$ is the direct product of a compact group $S$
and a vector group $A_h$. The compact group projects trivially to $A$ along $M$, 
hence $S\subset M$  
(whereas it need not be the case that $A_h\subset A$). This implies (\ref{identity component LcapH}).
The argument for (\ref{all of LcapH}) is the same. The last statement follows
from (\ref{all of LcapH}).
\end{proof}

\section{Finite multiplicity}

We assume throughout this section that $G/H$ is spherical
with $PH$ open, and that a Cartan involution $\theta$
of $G$ has been chosen as described in the preceding section.
Accordingly we write $P=MAN$. Let $\Sigma^+(\gf, \af)$ be the set
of positive roots of $\af$ in $\gf$, and let $\bar P=\theta(P)=MA\bar N$ be 
the opposite parabolic subgroup.

\par For a Harish-Chandra module $V$ we denote by $V^\infty$ its unique smooth 
moderate growth Fr\'echet completion. 
Note that  $V^\infty$ is a $G$-Fr\'echet module and we set 
$V^{-\infty}:= (V^\infty)'$ for its strong dual. Our goal is to provide a bound for the dimension of  
the space of $H$-invariants 
$$(V^{-\infty})^H:=  \mathrm{Hom}_H  (V^\infty, \C)$$
where $\mathrm{Hom}$ stands for continuous linear homomorphisms. 

\par For $v\in V^\infty$ and $\eta$ a continuous $H$-invariant functional on $V^\infty$ we form the 
matrix coefficient 
$$m_{v,\eta}(g):= \eta (\pi(g^{-1})v) \qquad (g\in G)\, .$$ 
Note that $m_{v,\eta}$ is a smooth function on $G$, even analytic  for $v\in V$.
We start with a general lemma which we prove later. We shall say that
$f:\R\to\C$ is of {\it super exponential decay} for $t\to\infty$ if
$f(t)=O(e^{\lambda t})$ for all $\lambda\in\R$.

\begin{lemma}\label{lem1} Let $V$ be a Harish-Chandra module.
Let $X\in \af$ be any element which is strictly anti-dominant with respect to $P$ and 
for which $\alpha(X)\in\Q$ for each $\alpha\in\Sigma(\gf,\af)$.
Fix $\eta\in (V^{-\infty})^H$. Suppose that 
for all $v\in V$ the function 
$$\R \ni t\mapsto m_{v,\eta} (\exp(tX)) \in \C$$
is of super exponential decay for $t\to \infty$. Then $\eta=0$.
\end{lemma}

\begin{proof} Let $v\in V$ and set $F_v(t)= m_{v,\eta} (\exp(tX))$.
We may assume that $\alpha(X)$ is an integer for each root $\alpha$.
With the new variable $z=e^{-t}$ we will show in Section \ref{appendix} 
that then $F_v$ admits an expansion 
\begin{equation}\label{power}  
F_v(t)=\sum_{j=1}^N \sum_{k=0}^M  z^{\lambda_j} (\log z)^k f_{j,k}(z),\quad(t\gg 0),
\end{equation}
where $\lambda_j\in \C$ and the $f_{j,k}$ are holomorphic functions in 
a neighborhood of $z=0$ in $\C$. The fact that $F_v$ is of superexponential 
decay then forces $F_v=0$ (this follows for example
from \cite{W}, Lemma 4.A.1.2).  Then 
$F_v(0)=\eta(v)=0$, and hence $\eta=0$ since $v\in V$ was arbitrary. 
\end{proof}

\par For a Harish-Chandra module $V$ we recall that $V/\bar\nf V$ is a finite dimensional module 
for $\bar P=MA\bar N$ with $\bar N$ acting trivially. 
Recall from Proposition \ref{lst} the
parabolic subgroup $Q=LU$ and its
subalgebra $\qf =\lf+\uf \supset \pf$. Let $\bar\qf=\theta(\qf)=\lf+\bar\uf$ and
define a subalgebra $\bar\qf_1$ of $\bar\qf$ by
\begin{equation} \label{eq1} 
\bar\qf_1:=(\lf\cap\hf) +{\bar\uf} \subset \hf + {\bar\nf}. \end{equation}
Note that $\bar\nf\subset \lf_n+\bar\uf \subset\qf_1$. Hence the
quotient $V/\bar\qf_1 V$ is finite dimensional. Moreover,
it carries a natural action of $L\cap K\cap H$ since $\bar\qf_1$ is $L\cap H$-invariant.

\begin{theorem}\label{thm1}  Let $V$ be a Harish-Chandra module and 
$V^\infty$ its unique smooth Fr\'echet completion. Then 
$$ \dim \mathrm{Hom}_H  (V^\infty, \C) \leq \dim (V/\bar\qf_1 V)^{L\cap K\cap H}\, .$$  
In particular {\rm (\ref{u-bound})} from the introduction is valid.
\end{theorem}

\begin{rmk}\label{Lcap H-rep}
If $L\cap H$ has finitely many components then
it follows from Lemma \ref{LcapH} that the trivial action of
$\lf\cap\hf$ on $V/\bar\qf_1 V$ lifts to an action of $L\cap H$ which
agrees with the natural action of $L\cap H\cap K$.
In this case then $(V/\bar\qf_1 V)^{L\cap K\cap H}=(V/\bar\qf_1 V)^{L\cap H}$ for this action.
\end{rmk}

\begin{proof}  
\par 
Set $\Upsilon=(V^{-\infty})^H$. It is almost immediate from the definitions 
(see \cite{KSS} eq.~(3.7)) that there exists $\delta\in \af^*$ 
such that for all $v\in V^\infty, \eta\in\Upsilon$ there is a constant $C_{v,\eta}>0$ 
such that 
\begin{equation} \label{ibound} 
|m_{v,\eta} (a)| \leq C_{v, \eta}  a^{\delta} \qquad (a\in {A^-}) \, .\end{equation} 

\par Fix an element $X\in \af$ as in Lemma \ref{lem1}. After rescaling, we may assume that 
$\min_{\alpha \in \Sigma^-(\gf, \af)}  \alpha(X)=1$.
For all $v\in V^\infty$ and $\eta\in\Upsilon$ we define 
\begin{equation*} 
%\label{power} 
F_{v,\eta}(t):= m_{v,\eta} (\exp(tX)) \qquad (t\geq 0)\,  .
\end{equation*} 

We say that $\Lambda\in \R \cup\{-\infty\}$ bounds $\eta\in\Upsilon$, 
provided for all $\Lambda'>\Lambda$ and  all $v\in V$, 
\begin{equation} \label{ibound2} 
\sup_{t\geq 0} \,e^{-t \Lambda'} |F_{v,\eta} (t)| <\infty \, .
\end{equation}
Let $\Upsilon_\Lambda\subset\Upsilon$ denote the space of elements bounded by $\Lambda$,
then $\Upsilon_{\Lambda_1}  \subset \Upsilon_{\Lambda_2}$ for $\Lambda_1\le\Lambda_2$.  
It follows from (\ref{ibound}) that 
\begin{equation}\label{Ups}
\Upsilon=\cup_\Lambda\Upsilon_\Lambda,
\end{equation}
and from Lemma \ref{lem1} that 
\begin{equation}\label{zero}
\Upsilon_{-\infty}=\cap_\Lambda\Upsilon_\Lambda=\{0\}
\end{equation}

The element $X$ acts on the space 
$(V/\bar\qf_1 V)^{L\cap K\cap H}$. First note that 
$\lf\cap\hf=\lf_n+\left((\mf+\af)\cap\hf\right)$
is normalized by $\af$. Hence so is $\bar\qf_1$ and thus
$X$ acts on $V/\bar\qf_1 V$.
Likewise, since $\af =\af\cap\zf(\lf) + \af\cap \hf$
we find that $\Ad(l)X=X \mod \lf\cap\hf$ for $l\in L\cap K\cap H$, and hence 
$X$ preserves the space of $L\cap K\cap H$-invariant vectors in the quotient. 

We denote by $\Xi$ the set of values $-\re\lambda$
where $\lambda\in\C$ is an eigenvalue for $X$ on
$(V/\bar\qf_1 V)^{L\cap K\cap H}$, and write
$\Xi=\{\mu_1,\dots,\mu_l\}$ where 
$$\mu_{l+1}:=-\infty<\mu_l<\dots<\mu_1<\mu_0:=+\infty.$$ 
%Furthermore, we lett $\mu_0=+\infty$ and $\mu_{l+1}=-\infty$.
Let $m_1,\dots, m_l$ denote
the sums of the algebraic multiplicities of the corresponding eigenvalues 
$\lambda$. Then 
\begin{equation*}
m_1+\dots+m_l=n_0:=\dim (V/\bar\qf_1 V)^{L\cap K\cap H}.
\end{equation*}
We shall prove:
\begin{enumerate}
\item\label{one}
%\begin{equation*}%\label{Ups}
$\Upsilon_\Lambda=\Upsilon_{\mu_{k+1}}$ for $\mu_{k+1}\le\Lambda<\mu_{k}$ and
$k=0,\dots,l$ 
%\end{equation*}
\item\label{two} The codimension
of $\Upsilon_{\mu_{k+1}}$ in $\Upsilon_{\mu_{k}}$ is  at most $m_{k}$ ($k=1,\dots,l$). 
\end{enumerate} 
It is easily seen that these statements together with (\ref{Ups}) and
(\ref{zero}) 
imply the theorem. Before proving (\ref{one}) and (\ref{two})
we need some preparations.

Let $\bar w_1, \ldots, \bar w_{n_0}$ be a basis for $(V/\bar\qf_1 V)^{L\cap K\cap H}$ 
and let $B$ denote the corresponding $n_0\times n_0$-matrix defined by 
$$X\bar w_j=\sum_{k} b_{jk} \bar w_k$$
for $j=1,\dots,n_0$.
We choose a representative $w_j \in V$ for each $\bar w_j$
and define 
$$u_j=Xw_j-\sum_k b_{jk}w_k\in \bar\qf_1 V.$$

We can arrange that  $B$ 
consists of block matrices $B_1,\dots,B_l$ along the diagonal, such that each
$B_k$ is an $m_k\times m_k$ matrix all of whose eigenvalues
have real part $-\mu_k$.
In the following we shall make the identification 
\begin{equation*}%\label{Cn0}
\C^{n_0}=\C^{m_1}\times \dots\times \C^{m_l}
\end{equation*}
and write elements $x\in\C^{n_0}$ accordingly
as $x=(x_1,\dots,x_l)$.  

For a given $\eta\in\Upsilon$ let 
$$F(t)=(F_1(t),\dots,F_l(t))\in \C^{n_0}$$
where $F_k(t)\in\C^{m_k}$ is the $m_k$-tuple with entries $F_{w_j,\eta}(t)$,
corresponding to the $k$-th block of $B$.
Likewise we put
$$
R(t)=(R_1(t),\dots,R_l(t))\in \C^{n_0}$$
where $R_k(t)$ has entries
$F_{u_j,\eta}(t)$.
Then for each $k$,
$$ F_k'(t) = -B_k  F_k(t) - R_k(t). $$
Hence
\begin{equation} \label{eqk} 
F_k(t) = e^{-tB_k}c_0- e^{-tB_k}\int_0^t e^{sB_k} R_k(s) \,ds\,
\end{equation}
where  $c_0=F_k(0)$.

\par  Next we recall from \cite{KSS}, proof of Thm.~3.2 (with
$\af^+,\bar\nf$ interchanged by $\af^-,\nf$), that the elements in
$\bar\nf V$ satisfy improved bounds as follows. 
Assume $\eta\in\Upsilon_\Lambda$
and $\Lambda'>\Lambda$. Then  
(\ref{ibound2}) is valid for all $v\in V$, and it follows 
for all $ u \in \bar\nf V$ that
\begin{equation}\label{ibound3} 
\sup_{t\geq 0} \, e^{-t (\Lambda'-1)} |F_{u,\eta} (t)| <\infty  . 
\end{equation}

The bound (\ref{ibound3}) is valid for $u\in \bar\qf_1V$ as well.
To see this we note first that since $\af =\af\cap\zf(\lf) + \af\cap \hf$ we have 
\begin{equation}\label{XY} 
m_{v,\eta}(\exp(tX))=m_{v,\eta}(\exp(tY)), \qquad(v\in V),
\end{equation}
for some $Y\in\af$ which centralizes $\lf$.
Then $u= \sum_j X_jv_j +u'$ with $X_j\in \lf\cap\hf$, $v_j\in V$ and 
$u' \in \bar\nf V$, and it follows that 
$YX_jv_j=X_jYv_j\in \hf V$. Hence $m_{u,\eta}(\exp(tY))=m_{u',\eta}(\exp(tY))$.
We conclude that
$$F_{u,\eta}=F_{u',\eta}$$
and hence (\ref{ibound3}) is valid as claimed.
We conclude the existence of a constant such that 
\begin{equation}\label{ibound4} 
|R_k(t)|\le Ce^{t (\Lambda'-1)}. 
\end{equation}

Based on (\ref{ibound4})
we shall provide the following two estimates for $F_k(t)$:
\begin{equation}\label{estimate}
\sup_{t\ge 0}\, e^{-t\gamma} |F_k(t)|<\infty,\quad 
\forall\gamma>\max\{\mu_k,\Lambda-1\}
\end{equation}
for all $k=1,\dots, l$, and 
\begin{equation}\label{better estimate}
\sup_{t\ge 0}\, e^{-t\gamma} |F_k(t)|<\infty,\quad \forall\gamma>\Lambda-1
\end{equation}
for those $k=1,\dots$ for which $\mu_k> \Lambda$.

Let $\gamma>\max\{\mu_k,\Lambda-1\}$ and
let $\Lambda'\in (\Lambda,\gamma+1)$.
It is clear that the first term in (\ref{eqk}) is bounded 
by a polynomial times $e^{t\mu_k}$. 
Applying (\ref{ibound4}) we see that the integrand 
$e^{sB_k} R_k (s)$ of the second term
is dominated by a polynomial times 
$e^{s(-\mu_k+\Lambda'-1)}$.
It follows that $|F_k(t)|$ is
is dominated by a polynomial times $e^{t\max\{\mu_k,\Lambda'-1\}}$,
and this implies (\ref{estimate}).

Before we prove the second estimate we note the following fact.
If $\mu_k>\Lambda'-1$ then
$e^{sB_k} R_k (s)$ is integrable to infinity. Hence if $\mu_k>\Lambda-1$,
we can replace
the solution formula
(\ref{eqk}) by
\begin{equation} \label{eqk1} 
F_k(t) = e^{-tB_k}c_\infty+ e^{-tB_k}\int_t^\infty e^{sB_k} R_k (s) \,ds
\end{equation}
where $c_\infty=\lim_{s\to\infty} e^{sB_k} F_k(s)$.
Note that with (\ref{ibound4}) the equation
(\ref{eqk1}) implies (\ref{better estimate}) 
for every $k$ for which the limit $c_\infty$ vanishes.

For the proof of (\ref{better estimate}) we assume
$\mu_k > \Lambda$.
Then it follows from (\ref{ibound2}) with $\Lambda<\Lambda'<\mu_k$ 
that $\lim_{s\to\infty} e^{sB_k} F_k(s)=0$. 
Hence $c_\infty=0$ and (\ref{eqk1}) implies
(\ref{better estimate}).

We are now ready to prove our claims (\ref{one}) and (\ref{two}).
%Let $\mu_0=+\infty$ and $\mu_{l+1}=-\infty$. 
Let $\Lambda\in\R$ and let
$\mu_\Lambda=\max\Xi\cap \R_{\le\Lambda}$ (with $\max\emptyset=-\infty$).
Let 
$\eta\in\Upsilon_\Lambda$. 
We will show 
\begin{equation}\label{Upsi}  
\eta\in\Upsilon_{\max\{\mu_\Lambda,\Lambda-1\}}.
\end{equation}
By iteration this will then imply that
$\Upsilon_\Lambda=\Upsilon_{\mu_\Lambda}$
as stated in (\ref{one}). Indeed, if $\Lambda-1\le\mu_\Lambda$
then this follows immediately.
Otherwise (\ref{Upsi}) implies
$\Upsilon_\Lambda=\Upsilon_{\Lambda-1}$. Since
$\Lambda$ was arbitrary we reach $\Upsilon_\Lambda=\Upsilon_{\mu_\Lambda}$ after
finitely many steps.

Let $\lambda'>\max\{\mu_\Lambda,\Lambda-1\}$. We need to prove
\begin{equation}\label{claim1} 
\sup_{t\ge 0}\, e^{-t\lambda'} |F_{v,\eta}(t)| < \infty
\end{equation}
for all $v\in V$.

Note that for $m\in L\cap K\cap H$ we have $F_{mv,\eta}=F_{v,\eta}$ (see (\ref{XY}))
and hence $F_{v,\eta}=F_{v_0,\eta}$, where $v_0$ is the projection
of $v$ to $V^{L\cap K\cap H}$. Hence for every $v\in V$ we have that $F_{v,\eta}$
is a linear combination the functions $F_{w_j,\eta}$ plus 
$F_{u,\eta}$ for some $u\in \bar\qf_0 V$. 
Thus it follows from (\ref{ibound3}) that it suffices to establish
(\ref{claim1}) for $v=w_j$ for all $j$, and hence it suffices to show
\begin{equation}\label{claim2} 
\sup_{t\ge 0}\, e^{-t\lambda'} |F_k(t)| < \infty
\end{equation}
for $k=1,\dots,l$. If $\mu_k\le\mu_\Lambda$ then this is 
clear from (\ref{estimate}).
On the other hand, if $\mu_k>\mu_\Lambda$ then $\mu_k>\Lambda$
and (\ref{claim2}) follows from (\ref{better estimate}).

Next we show (\ref{two}).
Let $1\le k_0\le l$ and $\eta\in \Upsilon_{\Lambda}$
where $\Lambda=\mu_{k_0}$. 
Then $\mu_{k_0}>\Lambda-1$ and according to 
(\ref{eqk1}) the limit
$\lim_{s\to\infty} e^{sB_{k}}F_{k}(s)$ exists for every $k\le k_0$.
Moreover, for $k< k_0$ we have $\mu_k>\Lambda$ and the limit vanishes
as seen in the proof of (\ref{better estimate}) above. 
If we assume that this limit is zero also for $k=k_0$,
then (\ref{better estimate}) is valid for all $k\leq k_0$,
from which we conclude as above that
(\ref{Upsi}) holds, with $\mu_\Lambda$ now replaced by $\mu_{k_0+1}$.
By step (\ref{one}) this implies 
that $\eta\in\Upsilon_{\mu_{k_0+1}}$. Hence
$$\Upsilon_{\mu_{k_0+1}}=\{\eta\in \Upsilon_{\mu_{k_0}}\mid \lim_{s\to\infty} 
e^{sB_{k_0}}F_{k_0}(s)=0\}$$
and as $\eta\mapsto \lim_{s\to\infty} e^{sB_{k_0}}F_{k_0}(s)$
is linear into $\C^{m_{k_0}}$, (\ref{two}) follows. 
\end{proof}

\section{The spherical subrepresentation theorem}\label{sphsub}

\par Let $Z=G/H$ be real spherical with $PH$ open, and recall
the definition (\ref{eq1}) of the subalgebra $\bar\qf_1\subset\bar\qf$.
The algebraic version of the 
subrepresentation theorem is:

\begin{cor}\label{C-lem} 
Let $V$ be a Harish-Chandra module and assume that $\mathrm{Hom}_H (V^\infty, \C)\neq \{0\}$. 
Then $(V/\bar\qf_1 V)^{L\cap K\cap H}\neq \{0\}$. 
\end{cor}

\begin{proof} Immediate from Theorem \ref{thm1}.  
\end{proof}

Let $\bar Q=\theta(Q)$ denote the parabolic subgroup opposite to $Q$, with
Levi decomposition $\bar Q=L\bar U$.
If $\tau$ is a finite dimensional representation of $\bar Q$, then we write 
$\mathrm{Ind}_{\bar Q}^G$ for the induced $(\gf, K)$-module of $K$-finite smooth sections 
of the $G$-equivariant vector bundle $\tau \times_{\bar Q} G \to \bar Q\bs G$.

\begin{theorem} 
Let $V$ be an irreducible Harish-Chandra module such that 
$\mathrm{Hom}_H (V^\infty, \C)\neq \{0\}$.  

Then there exists a finite 
dimensional irreducible representation 
$\tau$ of $\bar Q$, trivial on $\bar U$, such that
 $\mathrm{Hom}_{(\lf\cap\hf,L\cap H\cap K)}(\tau,\C)\neq \{0\}$, and an embedding
$$V \hookrightarrow \mathrm{Ind}_{\bar Q}^G \tau\, .$$
\end{theorem}

\begin{rmk}\label{second remark}
Assume that $L\cap H$ has finitely many connected components, and
recall (see Proposition \ref{lst}) that this assumption 
is valid for example when $G$ and $H$ are algebraic. Then
it follows from Lemma \ref{LcapH} that 
$\mathrm{Hom}_{L\cap H}(\tau,\C)\neq \{0\}$
for
the representation
$\tau$ above.
\end{rmk}

\begin{proof} 
Let $\bar\qf_0=\lf_n+\bar\uf$, then $\bar\qf_1=\bar\qf_0+(\lf\cap\hf)$.
The space $V/\bar\qf_0 V$ is
finite-dimensional since $\bar\nf\subset\bar\qf_0$,
and it carries compatible actions of
$\bar\qf$ and $L\cap K$ since $\bar\qf_0$ is an 
$L\cap K$-invariant ideal in $\bar\qf$.  Since $L$ is reductive
the action of the pair $(\lf,L\cap K)$ lifts uniquely 
to a representation of $L$ on $V/\bar\qf_0 V$.
We can then extend to an action of $\bar Q=L\bar U$, 
which is trivial on $\bar U$ and compatible
with the action of $\bar\qf$.

The quotient map
$V/\bar \qf_0V\to V/\bar\qf_1V$ is clearly a homomorphism for the pair
$(\lf\cap\hf,L\cap H\cap K)$. By Corollary \ref{C-lem} the module
$V/\bar \qf_1V$ has a non-zero vector fixed by the compact group
$L\cap K\cap H$. Hence also its dual admits such a vector,
and by composing with the quotient map we obtain that  
$\mathrm{Hom}_{(\lf\cap\hf,L\cap H\cap K)}(V/\bar\qf_0 V,\C)\neq \{0\}$.

Let $\tau$ be an irreducible $\bar Q$-subrepresentation
of $V/\bar\qf_0V$ for which  $\mathrm{Hom}_{(\lf\cap\hf,L\cap H\cap K)}(\tau,\C)\neq \{0\}$.
Since the quotient map $V/\bar \uf V\to V/\bar \qf_0V$ is $L$-equivariant, we have
$$\Hom_L(V/\bar \uf V,\tau)\neq \{0 \}.$$
As $V$ is irreducible, the desired embedding follows by
Frobenius reciprocity
(see \cite{HS}, Theorem 4.9, and note that our induction is not normalized). 
\end{proof}

\section{Regular singularities}\label{appendix}

The goal of this section is to provide a proof for the expansion 
(\ref{power}).  To begin with let us first recall that matrix coefficients on $Z$ satisfy 
certain systems of differential equations. For that we fix a Harish-Chandra module $V$ 
and a $K$-type $\tau$ occurring in $V$. We denote by $V[\tau]$ the $\tau$-isotypical part
of $V$, and consider for $\eta\in  (V^{-\infty})^H$ and $v\in V[\tau]$ 
the matrix coefficient
$$f(a)=m_{v, \eta}(a)= \eta(\pi(a^{-1})v)\, $$
on $A$ (where $A$ originates from the Iwasawa decomposition $G=KAN$
chosen in Section \ref{str}).

For simplicity we assume that $V$ is irreducible and obtain that 
the center $\Zc(\gf)$ of $\U(\gf)$ 
acts by scalars on $V$ 
(otherwise we replace the annihilating ideal of $V$ in $\Zc(\gf)$ 
with an ideal of finite co-dimension, and proceed as before).
The theory of $\tau$-radial parts (see Remark \ref{rmk2} below)
then gives a system of differential equations for $f$ on a
subcone of $A$. 

\par 
Let $R_1, \ldots, R_n$ be a basis of 
root vectors of $\uf$, say $R_j\in\gf^{\alpha_j}$ corresponds to the 
root $\alpha_j$. Set $Q_j:=\theta(R_j)\in \gf^{-\alpha_j}$. 
For $t\in\R$ we let 
\begin{equation*}%\label{At}
\af_t=\{ X\in\af \mid \alpha_j(X)<-t ,\, j=1,\dots,n\}
\end{equation*}
and $A_t:=\exp(\af_t)$. Let
$D_\e:=\{|z|<\e\}^n \subset \C^n$
for $\e>0$, and 
define
\begin{equation}\label{iota}
\iota: A_t \to \C^n , \ \ a\mapsto \iota(a)\:=(a^{\alpha_j})_{j=1,\dots,n}\, 
\end{equation}
where $a^{\alpha_j}=e^{\alpha_j(X)}$ for $a=\exp(X)$, 
then $\iota(A_t)\subset D_{e^{-t}}.$

\begin{lemma}\label{app1} There exists $t\ge 0$ such that 
for all $a\in A_t$ one has 
\begin{equation} \label{decomp} 
\gf= \Ad(a^{-1}) \kf +\af_Z +\hf \end{equation} 
\end{lemma} 

For symmetric spaces this is obtained with $t=0$ in \cite{vdB2} Lemma 1.5.  

\begin{proof} Recall the decomposition
(\ref{ldeco}). As $\mf_Z$ is centralized by $A$, we have
$\mf_Z\subset\Ad(a^{-1})\kf$ for every $a\in A$. Hence
it suffices to show that each $R_j\in\uf$ decomposes according to (\ref{decomp}).
For all $1\leq j \leq n$ we obtain from  (\ref{ldeco})
\begin{equation} \label{decomp5} 
Q_j= X_{\mf, j} + X_{\af, j} + X_{\hf, j} + \sum_{i=1}^n c_{ij} R_i
\end{equation}
with $X_{\mf, j} \in \mf_Z$, $X_{\af, j} \in \af_Z$ and $X_{\hf, j} \in \hf$. 
For $a\in A$ we have 
\begin{equation} \label{decomp3} R_j = a^{\alpha_j} \Ad(a^{-1})(Q_j + R_j)  - 
a^{2\alpha_j} Q_j\, ,\end{equation} 
and inserting (\ref{decomp5}) for the last term 
in (\ref{decomp3}) we obtain
\begin{equation} \label {decomp2} 
\begin{aligned}
R_j+  a^{2\alpha_j} &\sum_{i=1}^n c_{ij}R_i 
\\&=  
a^{\alpha_j} \Ad(a^{-1}) \underbrace{(Q_j +R_j)}_{\in \kf} 
-a^{2\alpha_j}(X_{\mf, j} + X_{\af, j} + X_{\hf, j}) \, 
\end{aligned}\end{equation} 
Consider the $n\times n$-matrix
\begin{equation*}%\label{decomp4} 
\1+ (c_{ij} a^{2\alpha_j})_{ij}.
\end{equation*}
It is invertible for $a\in A_t $ for sufficiently large $t$. 
This proves (\ref{decomp}).
\end{proof}

In fact, a more precise version of the decomposition (\ref{decomp}) is obtained as follows.
Let $(X_{j})_j$, $(Y_{k})_k$ 
and $(Z_{l})_l$ be fixed bases for $\kf$, $\af_Z$ and $\hf$,
respectively. Then the previous proof shows
(see (\ref{iota}) and (\ref{decomp2}))

\begin{lemma}\label{app2} %Let $X\in \oline{\nf}_Z$. 
There exists $t>0$ such that for every $X\in\gf$ there exist 
holomorphic functions 
$f_j, g_k, h_l \in \mathcal{O}(D_{e^{-t}})$
such that
for all $a\in A_t$ one has  
$$X= 
\sum_{j=1}^{\dim\kf} f_j(\iota(a)) \Ad(a^{-1})X_{j} +  
\sum_{k=1}^{\dim\af_Z} g_k(\iota(a))  Y_{k} + 
\sum_{l=1}^{\dim \hf } h_l(\iota(a) ) Z_{l}  .$$
Moreover, if $X\in\uf$, then this can be attained with
$$f_j, g_k, h_l \in \mathcal{O}_0(D_{e^{-t}})=\{\varphi\in \mathcal{O}(D_{e^{-t}})\mid \varphi(0)=0\}.$$
\end{lemma}

Let $(U_j)_j$ be a fixed homogeneous basis of $\U(\kf)$, 
$(V_k)_k$ one of $\U(\af_Z)$ and $(W_l)_l$ of $\U(\hf)$. 
We then obtain the following from Lemma \ref{app1} and 
Lemma \ref{app2}.

\begin{lemma}\label{app3} There exists $t>0$ such that
for every $u\in \U(\gf)$ there exist $f_{j,k,l}\in \mathcal{O}(D_{e^{-t}})$
such that
\begin{equation} \label{deco} 
u = \sum_{j,k,l} 
f_{j,k,l} (\iota(a)) (\Ad(a^{-1}) U_j) V_k W_l ,\end{equation} for all $a\in A_t$,
where $j,k,l$ extend over all sets of indices with
$$\deg (U_j) + \deg (V_k) +\deg(W_l)\leq \deg(u).$$ 
\end{lemma}
 
\begin{proof} Let $t$ be as in Lemma \ref{app2}. We proceed by induction on $\deg(u)$,
the case of degree zero being clear.
Let $u\in \U(\gf)$ and assume $u=Xv$ with $X\in\gf$ and $v\in \U(\gf)$ of degree one less.
Applying Lemma  \ref{app2} to $X$, we see that
it suffices to treat the case where $X\in\af_Z+\hf$. We write
$Xv=[X,v]+vX$ and apply the induction hypothesis to $[X,v]$. 
For the term $vX$ we
apply the induction hypothesis to $v$ and write it as a 
linear combination of elements $(\Ad(a^{-1}) U_j) V_k W_l$
with holomorphic coefficients and $\deg(U_j)+\deg(V_k)+\deg(W_l)\le\deg(v)$. Next we write
$$(\Ad(a^{-1}) U_j) V_k W_l X=(\Ad(a^{-1}) U_j) V_k([W_l,X]+XW_l).$$ 
The induction hypothesis applies to $V_k[W_l,X]$. The terms with $V_kXW_l$
already have the form asserted in (\ref{deco}).
\end{proof}

\begin{rmk}\label{rmk2} The decomposition (\ref{deco}) allows one to define the $\tau$-radial part of an 
element $u\in \U(\gf)^{M\cap H}$. This goes as follows: 
We let $\U(\gf)$ act on $C^\infty(G)$ by right 
differentiation: 
$$ (X\cdot f)(g) = \frac d{dt}\Big|_{t=0} f(g\exp(tX))\qquad (f\in C^\infty(G), g\in G, X\in \gf)\, .$$
Henceforth we regard smooth functions on $Z$ as right $H$-invariant functions on $G$. 
For a $K$-type $(\tau,W_\tau)$
we denote by $C^\infty(Z)_\tau$ the space of functions which are of left $K$-type $\tau$, for example the
matrix coefficients $m_{v,\eta}$ with $v$ in $V[\tau]$.  
\par It is easily seen from the Peter-Weyl theorem
that for every $f\in C^\infty(Z)_\tau$ there exist a unique 
$\Phi\in C^\infty(Z, W_\tau\otimes W_{\tau^*})$ 
such that
$$f(gH)= \Tr (\Phi(gH)),\quad gH\in Z,$$
where $\Tr$ stands for contraction $W_\tau\otimes W_{\tau^*}\to\C$.
Moreover $\Phi(kgH)=(1\otimes \tau^*(k))\Phi(gH)$ for $k\in K$.
In particular, $\Phi(aH)\in W_\tau\otimes W^{M\cap H}_{\tau^*}$ for $a\in A$.
Hence we have a finite sum
\begin{equation}\label{Fjsum}
f(gH)=\sum_j \la F_j(gH),w_j\ra,
\end{equation}
with $F_j\in C^\infty(Z,W_{\tau^*})$ and $w_j\in W_\tau$.
Moreover, $F_j(kgH)=\tau^*(k)(gH)$ for all $k\in K, g\in G$.

\par Let $u\in \U(\gf)^{M\cap H}$ and let $t>0$ 
be as above. According to (\ref{deco}) we can write $u$ as a sum of 
$(\Ad(a^{-1})U) V W$ with $U\in \U(\kf)^{M\cap H}$, $V\in \U(\af_Z)$ and $W\in \U(\hf)$
and coefficients depending holomorphically on $\iota(a)$ for $a\in A_t$. 
In order to compute $ u\cdot f|_{A_t} $ we may assume that $W=\1$ as
$f$ is right $H$-invariant. For $u=(\Ad(a^{-1}))UV$ we then have
$$(u\cdot f)(a) = \sum  \la  (V\cdot F_j)(a), U^t\cdot w_j\ra $$  
We finally arrive at an action of $\U(\gf)^{M\cap H}$ on $C^\infty (A_t, 
W_{\tau^*}^{M\cap H})$ given by 
$${\rm rad}(u)(F)(a):= \sum \tau^*(U) (V\cdot F)(a)\,,
$$
a differential operator with $\End(W_{\tau^*}^{M\cap H})$-valued coefficients.
With this definition we find
$$uf(a)=\sum_j \la\operatorname{rad}(u)(F_j)(a),w_j\ra$$
when $f$ is given by (\ref{Fjsum}).

\par In particular for $f=m_{w,\eta}$ 
the functions $F_j \in C^\infty(A_t, W_{\tau^*}^{M\cap H})$ obtained as above
satisfy the system of differential equations
\begin{equation*}%\label{rad-sys} 
{\rm rad }(z)(F) = \chi_V(z) F
\qquad (z\in \Zc(\gf))\end{equation*}
with $\chi_V: \Zc(\gf)\to \C$ the infinitesimal character of $V$.  
\end{rmk}

In the decomposition (\ref{deco}) of $u\in \U(\gf)$ we would like to restrict 
the middle parts $V_k$ from $\U(\af_Z)$ to a fixed finite set of elements, independent of
$u$. This will be done at the cost of enlarging the product with an extra factor from $\Zc(\gf)$
(which act by scalars on $V$).

For that we need to recall  parts of the  construction of the Harish-Chandra homomorphism. 
The decomposition $\gf=\nf +\af +\mf +\bar \nf$ results in a 
direct sum decomposition 
$$\U(\gf) = (\nf\, \U(\gf) + \U(\gf)\bar\nf) + \U(\af+ \mf)$$
and allows for a linear projection 
$$\mu_1: \U(\gf) \to \U(\af+ \mf)\, .$$
Recall (\cite{vdB} Lemma 3.6) that $\mu_1$ restricts to an algebra
homomorphism $\Zc(\gf)\to\Zc(\mf+\af)$ and that for $z\in \Zc(\gf)$ 
of degree $d$ one has 
$$z-\mu_1(z)\in \nf \,\U(\gf)_{d-1},$$
(where $\U(\gf)_{d-1}$ signifies elements of degree $\le d-1$).
It follows from the Harish-Chandra isomorphism theorem that  
$\Zc(\mf+\af)$ is finitely generated over $\mu_1(\Zc(\gf))$.
More precisely (see \cite{vdB}, Lemma 3.7), there exist
elements $v_1,\dots,v_r\in\Zc(\mf+\af)$ such that every
$v\in\Zc(\mf+\af)$ can be written as
\begin{equation}\label{vsum}
v=\sum_{j=1}^r \mu_1(z_j)v_j
\end{equation}
with $z_j\in \Zc(\gf)$ and
$\deg(z_j)+\deg(v_j)\le\deg(v)$ for each~$j$.

Further, as $\af$ and $\mf$ commute and as
$\af=\af_Z+(\af\cap\hf)$,
we obtain an algebra homomorphism 
$$p: \U(\mf +\af) \to \U(\af_Z)$$ with
kernel $(\mf+(\af\cap\hf))\,\U(\mf+\af)$.  
We compose with $\mu_1$ and obtain a linear map
$$\mu_2: \U(\gf)\to \U(\af_Z)\, .$$ 
The restriction 
to $\Zc(\gf)$ is an algebra homomorphism and 
will be denoted by $\mu$. It follows from the above that
\begin{equation}\label{zdeco}
z-\mu(z)\in \nf\, \U(\gf)_{d-1} + (\mf+(\af\cap\hf))\,\U(\mf +\af)_{d-1}
\end{equation} 
for $z\in\Zc(\gf)$.
Furthermore by applying $p$ to (\ref{vsum}) we see that
$\U(\af_Z)$ is finitely generated over $\mu(\Zc(\gf))$
with generators $p(v_j)\in \U(\af_Z)$. Let $\Yc$ denote the finite set
$$\Yc=\{p(v_1),\dots,p(v_r)\}\subset\U(\af_Z).$$

\par 
\begin{lemma} \label{app4}For all $n\in \Z_{\ge 0}$ there exists $t=t_n>0$ such 
that for all $u\in \U(\gf)$ with $\deg(u)\le n$ 
there exist 
$$U_j \in \U(\kf), 
V_j \in \Yc\subset\U(\af_Z),  W_j \in 
\U(\hf), z_j\in\Zc(\gf)$$  
with 
$$\deg U_j +\deg V_j + \deg W_j+\deg z_j \leq n,$$
and holomorphic functions $f_j \in \mathcal{O}(D_{e^{-t}})$ such that 
\begin{equation}\label{sum with Z}  
u=\sum_{j=1}^p f_j(\iota(a)) 
(\Ad(a^{-1})U_j)V_{j} W_j z_j\end{equation} 
for all $a\in A_{t}$.  
\end{lemma}

\begin{proof} Note the dependence of $t$ on the degree of $u$, contrary to
what was the case in Lemma \ref{app3}. 
The proof will again be by induction on $n$, and the case $n=0$ is again clear. 
Furthermore, we see from Lemma \ref{app3} that it suffices to establish
the decomposition for $u\in\U(\af_Z)_n$, since if $\deg(U_j)>0$ or
$\deg(W_l)>0$ in (\ref{deco}) then $\deg(V_k)<n$ and the induction hypothesis applies.

Using (\ref{vsum}) 
we write $u=\sum \mu(z_{i}) Y_i$ with $z_{i}\in \Zc(\gf)$ and $Y_i\in\Yc$,
such that 
$\deg(z_{i}) + \deg(Y_i)\leq n$ for all $1\leq i \leq r$. 
With $\nf=\uf+\nf_L$ and $\nf_L\subset \hf$ 
we obtain from  (\ref{zdeco}) for each $i$,
$$\mu(z_{i}) =z_{i} +  \sum_m \underbrace{ R_m}_{\in\uf} u_{mi}  + 
\sum_{m'} \underbrace{S_{m'}}_{\in\mf} {u_{m'i}'} +  
\sum_{m''} \underbrace{T_{m''}}_{\in \hf} 
u_{m''i}''$$
with $u_{mi},u_{m'i}',u_{m''i}''\in\U(\gf)$  all
of degree $< \deg z_{i}$.  
Hence
$$u=\sum_i z_{i}Y_i +  \sum_{i,m}  R_m u_{mi}Y_i  + 
\sum_{i,m'} S_{m'} u_{m'i}'Y_i +  
\sum_{i,m''} T_{m''} u_{m''i}''Y_i.$$
The terms in the first sum already have the desired form. 
The terms with $S_{m'}$ are dealt with directly by the induction hypothesis,
and the terms with $T_{m''}$ are dealt with similarly
after commuting $T_{m''}$ and $u_{m''i}''Y_i$. 
Hence only the second sum remains, which is a sum of terms
$Rv$ with $R\in\uf$ and $v\in\U(\gf)_{n-1}$.

We use Lemma \ref{app2} to decompose $R$ as a linear combination of
terms $\Ad(a^{-1})X$, with $X\in\kf$, and basis vectors of $\af_z$ or $\hf$, 
and with coefficients from $\mathcal{O}_0(D_{e^{-s}})$ (for some fixed $s>0$). 
By the induction hypothesis applied to $v$, 
the terms $(\Ad(a^{-1})X)v$ again have the desired form (\ref{sum with Z}). 
Here $a\in A_{t_{n-1}}$.
We thus reach the conclusion that for each such $a$, our element
$u\in\U(\af_Z)_n$ is of form (\ref{sum with Z})
plus elements of the form $g(\iota(a)) w$ 
with $g\in \mathcal{O}_0(D_{e^{-s}})$ and $w\in\U(\gf)_n$ independent of $a$. 
The same conclusion then applies to every element $u\in\U(\gf)_n$,
as observed in the beginning of the proof.

\par We do the above for a basis $u_1, \ldots , u_N$ of $\U(\gf)_n$ and 
finally use the fact that an $N\times N$-matrix 
of the form $\1_{N\times N} + F(z)$ with $F$ holomorphic and $F(0)=0$ is invertible for 
sufficiently small $z$.
\end{proof} 
  
We can now give the proof of (\ref{power}). 
For $v\in V[\tau]$ we consider the matrix coefficient
$f(gH)=m_{v,\eta}(gH)=\eta(\pi(g)^{-1}v)$,
which we recall is a joint eigenfunction for $\Zc(\gf)$. As explained
in Remark \ref{rmk2}, $f$ is a sum of functions
of the form $\la F(gH),w\ra$ where $w\in W_\tau$ and where
$F\in C^\infty(Z, W_{\tau^*})$ is $\tau^*$-spherical, that is,
$F(kgH)=\tau^*(k)F(gH)$ for all $k\in K$, $g\in G$.
Note that this implies
$F(aH)\in W_{\tau^*}^{M\cap H}$ for $a\in A$.
Furthermore, the action of
elements from $\U(\gf)^{M\cap H}$ is computed by taking radial parts,
and $F$ is a joint eigenfunction for $\Zc(\gf)$.

Let $\Pi\subset\af^*$ denote the set of simple roots. For simplicity
we assume $\gf=[\gf,\gf]$ in the remainder of 
this section, so that
$\Pi$ is a basis for $\af^*$.
Extension to the reductive case is
elementary. 
Let $d=|\Pi|=\dim\af$.

\newcommand{\Dc}{\mathcal{D}}
\begin{proposition}
Let $F\in C^\infty(Z, W_{\tau^*})$ be a $\tau^*$-spherical
joint eigenfunction for $\Zc(\gf)$.
There exist a neighborhood $\Dc$ of $0$ in $\C^d$, a number $M\in\N$, a finite set
$S\subset \C^d$,
and for each $s\in S$ and each multiindex $0\le |m|\le M$ 
a holomorphic $W_{\tau^*}^{M\cap H}$-valued function
$h_{s,m}$ on $\Dc$ such that
\begin{equation}
F(aH)=\sum_{s\in S}\sum_{0\le |m|\le M} z^s(\log z)^m h_{s,m}(z),
\end{equation}
for all $a\in A$ such that 
 $z=(a^{\alpha})_{\alpha\in\Pi}\in \Dc$.
\end{proposition}

\begin{proof}
Let 
$X_1,\dots,X_d$ be the basis 
for $\af$ which is dual to $\Pi$,
and
let $D$ be the maximal degree of the 
the finite set of operators from $\Yc\subset\U(\af_Z)$.
For each multi-index $k=(k_1,\dots,k_d)$ with $|k|\le D$
we define
$$F_k(gH)=X_1^{k_1}\cdots X_d^{k_d}F(gH),\quad gH\in Z,$$
a $W_{\tau^*}$-valued function. 
For each $j=1,\dots,d$ and each multiindex $l$ with $|l|\le D$, the function $X_j F_l$ can then be determined
by giving the monomial $X_jX_1^{k_1}\cdots X_d^{k_d}\in\U(\af)$ an 
expression (\ref{sum with Z}) according to Lemma \ref{app4} with $t=t_{D+1}$,
and applying the theory of radial parts, as explained in
Remark \ref{rmk2}.
It follows that on $A_t$, the function $X_jF_l$ is 
a combination of the $F_k$'s, with $|k|\le D$ and with
$\End(W_{\tau^*}^{M\cap H})$-valued coefficients depending on $a\in A_t$. 

We thus see that the vector-valued function $\mathbf F(a)=(F_k(aH))_k$
satisfies a first order ordinary differential system 
$$X_j\mathbf F (a)= M(a) \mathbf F(a)$$
on $A_t$ with an operator $M(a)$. Furthermore, it follows from
(\ref{sum with Z}) that $M(a)$ extends to a holomorphic function of 
 $z=(a^{\alpha_j})_{j=1,\dots,n}\in\C^n$
in a neighborhood of $0$. In particular, $M(a)$ is then a 
holomorphic function of the smaller tuple 
 $(a^{\alpha})_{\alpha\in\Pi}\in\C^d$
in a neighborhood of $0$, since every positive root is a sum of simple roots. 
With coordinates on $A$ given by 
 $(a^{\alpha})_{\alpha\in\Pi}\in\C^d$
we thus obtain a system of equations,
which has a simple singularity at $0$ according to
\cite{Kn}, p 702, Example 2. The proposition now follows from
 \cite{Kn}, Thm.~B.16.
\end{proof}

By taking inner products with elements from $W_\tau$ we obtain
a similar expansion of the matrix coefficients $m_{v,\eta}(a)$ for $a\in A_t$,
for all $v\in V[\tau]$.

\begin{cor}
Let $\eta\in (V^{-\infty})^H$ and $v\in V[\tau]$.
There exist a neighborhood $\Dc$ of $0$ in $\C^d$, a number $M\in\N$, a finite set
$S\subset \C^d$,
and for each $s\in S$ and each multiindex $0\le |m|\le M$ 
a holomorphic function
$f_{s,m}$ on $\Dc$ such that
\begin{equation}
m_{v,\eta}(a)=\sum_{s\in S}\sum_{0\le |m|\le M} z^s(\log z)^m f_{s,m}(z),
\end{equation}
for all $a\in A$ such that 
 $z=(a^{\alpha})_{\alpha\in\Pi}\in \Dc$.
\end{cor}

Let now $X\in\af$ be such that $\alpha(X)\in\Z_-$ 
for each $\alpha\in\Pi$ and consider $m_{v,\eta}(a_s)$ for $s\to\infty$,
where $a_s=\exp(sX)$. Note that  $(a_s^{\alpha})_{\alpha\in\Pi}\in \Dc$
for $s\gg 0$ and that 
 $a_s^{\alpha}$ is a positive integral
power of $e^{-s}$ for each $\alpha\in\Pi$.
Hence $f_ {s,m}( (a_s^{\alpha})_{\alpha\in\Pi})$ extends to a holomorphic
function of $\zeta=e^{-s}$.
This completes the proof of 
(\ref{power}) and hence of
Lemma \ref{lem1}.

\end{document}